\documentclass[a4paper,10pt]{article}
\newcommand{\colorprint}[1]{}
\usepackage[utf8]{inputenc}
\usepackage{color}
\usepackage{tikz}
\usetikzlibrary{shapes}
\usepackage{float}
\usepackage{subcaption}
\usepackage{amsmath}
\usepackage{amsfonts}
\usepackage{amssymb}
\usepackage{amsthm}
\usepackage{mathrsfs}
\usepackage{amsthm}
\usepackage{subcaption}
\usepackage{graphics}
\usepackage[colorlinks,urlcolor=magenta,linkcolor=blue,citecolor=teal]{hyperref}
\theoremstyle{plain}
\newtheorem{theorem}{Theorem}
\newtheorem{prop}[theorem]{Proposition}
\newtheorem{lemma}[theorem]{Lemma}
\newtheorem{corollary}[theorem]{Corollary}
\newtheorem{problem}[theorem]{Problem}
\newtheorem{conjecture}[theorem]{Conjecture}
\newcommand{\dom}{\rightarrow}

\newcommand{\AY}[1]{#1}
\newcommand{\jbj}[1]{{#1}}
\newcommand{\jbjb}[1]{{#1}}

\newcommand{\LineBLUE}[1]{\draw [color=blue, line width=0.03cm] #1}
\newcommand{\LineRED}[1]{\draw [color=red, line width=0.03cm] #1}
\newcommand{\LineBLUEdotted}[1]{\draw [color=blue, dotted, line width=0.03cm] #1}
\newcommand{\LineREDdotted}[1]{\draw [color=red, dotted, line width=0.03cm] #1}

\tikzstyle{vertexX}=[circle,draw, top color=gray!10, bottom color=gray!70, minimum size=14pt, scale=0.6, inner sep=0.1pt]
\tikzstyle{vertexY}=[circle,draw, top color=green!10, bottom color=green!70, minimum size=14pt, scale=0.6, inner sep=0.1pt]
\tikzstyle{vertexZ}=[circle,draw, top color=orange!10, bottom color=orange!70, minimum size=14pt, scale=0.6, inner sep=0.1pt]

\begin{document}
\bibliographystyle{plain}
\title{Supereulerian 2-edge-coloured graphs\thanks{Research supported by the Independent Research Fund Denmark under grant number DFF 7014-00037B.}}
\author{J. Bang-Jensen\thanks{Department of Mathematics and Computer Science, University of Southern Denmark, Odense, Denmark (email: jbj@imada.sdu.dk). 
Part of this work was done while the author was on sabbatical at  INRIA Sophia Antipolis. Hospitality and financial support is greatefully acknowledged. Ce travail a b\'en\'efici\'e d'une aide du gouvernement français, 
g\'er\'ee par l'Agence Nationale de la Recherche au titre du projet Investissements d’Avenir UCAJEDI portant la r\'ef\'erence no ANR-15-IDEX-01.}
    \and Thomas Bellitto\thanks{Faculty of Mathematics, Informatics and Mechanics, University of Warsaw, Poland / Department of Mathematics and Computer Science, University of Southern Denmark, Odense, Denmark. This author is also supported by the European Research Council (ERC) under the European Union’s Horizon 2020 research and innovation program Grant Agreement 7147 (email: thomas.bellitto@mimuw.edu.pl)}
\and A. Yeo\thanks{Department of Mathematics and Computer Science, University of Southern Denmark, Odense, Denmark (email:yeo@imada.sdu.dk) and
Department of Pure and Applied Mathematics, University of Johannesburg, Auckland Park, 2006 South Africa}}

  \maketitle

  \begin{abstract}

   A 2-edge-coloured graph $G$ is {\bf supereulerian} if $G$ contains a spanning closed trail in which the edges alternate in colours. An {\bf eulerian factor} of a 2-edge-coloured  graph is a collection of vertex disjoint induced subgraphs which cover all the vertices of $G$ such that each of these subgraphs is supereulerian.
We give a polynomial algorithm to test if a 2-edge-coloured graph has an eulerian factor and to produce one when it exists.
A 2-edge-coloured graph is {\bf (trail-)colour-connected} if it contains a pair of alternating $(u,v)$-paths ($(u,v)$-trails) whose union is an alternating  closed walk for every pair of distinct vertices $u,v$. 
A 2-edge-coloured graph is {\bf M-closed}  if $xz$ is an edge of $G$ whenever some vertex $u$ is joined to both $x$ and $z$ by edges of the same colour. M-closed 2-edge-coloured graphs, introduced in \cite{balbuenaDMTCS21}, form  a rich generalization of 2-edge-coloured complete graphs.
We show that if $G$ is an extension of an M-closed 2-edge-coloured complete graph, then $G$ is supereulerian if and only if $G$ is trail-colour-connected and  has an eulerian factor.
We also show that for general 2-edge-coloured graphs it is NP-complete to decide whether the graph is supereulerian. Finally we pose a number of open problems.

\noindent{}{\bf Keywords:} 2-edge-coloured graph; alternating hamiltonian cycle; supereulerian; alternating cycle; eulerian factor; extension of a 2-edge-coloured graph.


  \end{abstract}

  \section{Introduction}
  \jbjb{In this paper the graphs we deal with may contain  parallel edges. For readability we will use the word graphs instead of the more correct name multigraphs. All our results are valid for multigraphs.}
Edge-coloured graphs form a very interesting generalization of (directed) graphs, a fact that has been used many times in the literature (see e.g. \cite{bangDM165,bangDM188,haggkvistC9} and \cite[Chapter 16]{bang2009}). As an example consider the conversion of a given digraph $D=(V,A)$ to a 2-edge-coloured bipartite graph $G(D)$: The vertex set of $G(D)$ is $V\cup \{w_{uv}|uv\in A\}$ and the set of edges of $G(D)$  consist of and edge $uw_{uv}$ of colour 1 and an edge $w_{uv}v$ of colour 2 for every arc $uv\in A$. It is easy to see that every directed path, cycle, trail and walk, respectively in $D$ corresponds to a path, cycle, trail and walk, respectively in $G(D)$ where the colours alternate between 1 and 2. The converse also holds when the path, trail or walk must start and end in a vertex of $V$. 

In Section \ref{sec22} we describe another correspondence (the BB-correspondence) between bipartite digraphs and bipartite 2-edge-coloured graphs which immediately implies that it is NP-complete to decide whether a 2-edge-coloured graph has a hamiltonian cycle whose edges alternate in colours (we call such a cycle {\bf alternating}), see Theorem \ref{thm:npcbip2ec}. Thus it makes sense to identify classes of 2-edge-coloured graphs for which one can solve problems such as alternating hamiltonian cycle, longest alternating cycle etc. in polynomial time. This has been the topic of many papers in the past, see e.g
\cite{bangDM188, bankfalvi1968,benkouarLNCS557,chetwyndJGT16,balbuenaDAM229,balbuenaDMTCS21,dasDM43,liDM340,liGC35,manoussakisDAM56,saadCPC5}.\\

The main topic of this paper is supereulerian 2-edge-coloured graphs, that is, 2-edge-coloured graphs that have a spanning alternating closed trail. 
Note that we also want the first and last edge of an alternating closed trail to have different colours.
In order to obtain our results we also derive new results on 2-edge-coloured graphs with an alternating hamiltonian cycle.
Bankfalvi and Bankfalvi \cite{bankfalvi1968} obtained a characterization of 2-edge-coloured complete graphs with an alternating hamiltonian cycle
(for an equivalent formulation of their result see Theorem \ref{thm:HCchar}), thus answering a question of Erd\"os.  
Das and Rao \cite{dasDM43} generalized this result  to spanning closed trails with prescribed degrees at every vertex, that is, we are given an even  positive number $f(v)$ for each vertex of the graph and  seek a spanning closed trail $T$ such that every vertex $v$ has degree exactly $f(v)$ in $T$. Bang-Jensen and Gutin \cite{bangDM188}  gave a polynomial algorithm for the more general problem of finding a longest alternating trail $T_{f\geq}$ in a 2-edge-coloured complete graph that visits each vertex $v$ at most $f(v)>0$ times. When $f(v)=2$ this solves the longest alternating cycle problem, previously solved by Saad \cite{saadCPC5}, and the problem solved by Das and Rao has a solution if and only if the length of $T_{f\geq}$ is exactly $\frac{1}{2}\sum_{v\in V}f(v)$.  None of the results above answer the question of when a 2-edge-coloured complete graph is supereulerian.  The key tool in \cite{bangDM188} is to study longest alternating cycles in extensions of 2-edge-coloured complete graphs (defined below). In order to be able to use a similar approach  we first show how the characterization of hamiltonian M-closed 2-edge-coloured graphs in \cite{balbuenaDMTCS21} can be extended to  a characterization of those extensions of M-closed 2-edge-coloured graphs which have an alternating hamiltonian cycle (Theorem \ref{thm:extMclHC}). We then show how this characterization can be used to derive a characterization of supereulerian extensions of M-closed 2-edge-coloured graphs (Theorem \ref{thm:exMclEchar}).

We also show that, as it is the case for graphs \cite{pulleyblankJGT3}  and digraphs \cite{bangTCS526}  for general 2-edge-coloured graphs it is NP-complete to decide whether the input is supereulerian. Finally we consider another generalization of 2-edge-coloured complete graphs, namely 2-edge-coloured complete multipartite graphs.

  \section{Notation and Preliminaries}

  Notation not defined here will be consistent with \cite{bang2009}.

  \jbjb{In this paper, whenever we talk about a 2-edge-coloured graph, we will assume that its edges are coloured by colours 1 and 2. In figures we will use red and blue edges instead of numbers 1 and 2.}
Let $G=(V,E)$ be a graph and let $\phi:E\rightarrow \{1,2\}$ be a 2-edge-colouring of $E$.
A path, cycle,trail or walk $X$ in $G$ is {\bf alternating} if the edges of $X$ alternate between colours 1,2.

\subsection{Colour-connectivity}

The graph $G$ is {\bf colour-connected} if there exist two alternating $(u,v)$-paths $P_1,P_2$ whose union is an alternating walk for every choice of distinct vertices $u,v$.

\begin{lemma}\cite{bangDM188}
One can decide in polynomial time whether a given 2-edge-coloured graph is colour-connected.
\end{lemma}

\begin{lemma}\label{lem:ccsuff}
Let $G$ be a 2-edge-coloured graph. Then $G$ is colour-connected if and only if $G$ has  an alternating  $(u,v)$-path starting with colour $c$ for each colour $c\in\{1,2\}$ and  every ordered pair of vertices $u,v$.
   \end{lemma}
   
   \begin{proof}
     By the definition of colour-connectivity, if $G$ is colour-connected, then it has the desired paths.
    Assume now that $G$ has  an alternating  $(u,v)$-path starting with colour $c$ for each colour $c\in\{1,2\}$ and every ordered pair of vertices $u,v$. Let $u$ and $v$ be vertices of the graph. 
Then, there exists two alternating paths $P_1$ and $P_2$ from $u$ to $v$ starting with colours 1 and 2 respectively. If they both end on different colours, then the union of $P_1$ and $P_2$ form the alternating closed walk required in the definition of colour-connectivity. Otherwise, let us assume by symmetry that $P_1,P_2$  both end with colour 1. We also know that there exists an alternating path $P_3$ from $v$ to $u$ starting with colour 2. If $P_3$ ends with colour 1, then $P_3$ and $P_2$ meet the requirement for colour-connectivity. If $P_3$ ends with colour 2, the requirement is met by $P_3$ and $P_1$.
\end{proof}

Let $G$ be a 2-edge-coloured graph on $n>1$ vertices $\{v_1,v_2,\ldots{},v_n\}$. By an {\bf extension} of  $G$ we mean any 
graph $H=G[I_{p_1},\ldots{},I_{p_n}]$ that is obtained from $G$ by replacing each vertex $v_i$ by an independent set $\{v_{i,1},\ldots{},v_{i,p_i}\}$ of $p_i\geq 1$ vertices, $i\in [n]$ and connecting different such sets as follows: If $v_iv_j$ is an edge in $G$ of colour $c$ then $H$ contains an edge of colour $c$ between $v_{i,q}$ and $v_{j,r}$ for every choice of $q\in [p_i], r\in [p_j]$. 
The following proposition turns out to be useful.


\begin{prop}\label{prop:extcc}
For a 2-edge-coloured graph $G$  the following are equivalent.
\begin{itemize}
\item[(i)] $G$ is colour-connected.
\item[(ii)] Every extension $H$ of $G$ is colour-connected.
\end{itemize}
\end{prop}

\AY{
\begin{proof}
Clearly if every extension of $G$ is colour-connected then $G$ is also colour-connected, as $G$ is an extension of itself.
We therefore just need to prove that (i) implies (ii).  Let $G$ be a colour-connected  2-edge-coloured graph with 
 $V(G)=\{v_1,v_2,\ldots{},v_n\}$.
Let $H=G[I_{p_1},\ldots{},I_{p_n}]$ be an extension of $G$ and let $u,v \in V(H)$ be arbitrary and assume that $u \in I_a$ and $v \in I_b$.
If $a \not= b$, then any path from $v_a$ to $v_b$ in $G$ gives rise to a path from $u$ to $v$ in $H$, so, as $G$ is color-connected, 
there exists alternating paths $P_1$ and $P_2$ from $u$ to $v$ in $H$ such $P_1$ starts with colour $1$ and $P_2$ starts with colour $2$.

So now consider the case when $a=b$. Let $v_a v_s$ be an edge in $G$ of color $1$. Let $Q_1$ be an alternating path from $v_s$ to $v_a$ starting with color $2$ (which exists as $G$ is color-connected). The path $Q_1$ corresponds to a path from $I_a$ to $v$ and adding $u$ to the front of this path we obtain a path $P_1$ from $u$ to $v$ starting with colour 1.  Analogously, we can also find a path $P_2$ from $u$ to $v$ starting with colour $2$.

By Lemma~\ref{lem:ccsuff} we have proven that $H$ is colour-connected.
\end{proof}
}

A graph $G$ is {\bf complete multipartite} if its vertices can be covered by a collection of disjoint independent sets $X_1,\ldots{},X_k$, for some $k\geq 2$, such that each pair $X_i,X_j$ with $i\neq j$ form a complete bipartite graph.

\subsection{Alternating cycles in 2-edge-coloured graphs}\label{sec22}

In this section we recall some results on alternating hamiltonian cycles.

An {\bf alternating cycle factor} in a 2-edge-coloured graph $G$ is a collection of disjoint alternating cycles that cover  $V(G)$

We start by  recalling a very useful correspondence between bipartite 2-edge-coloured graphs and directed bipartite graphs. 
This has been used several times in the literature, see e.g. 
\cite{bangJGTconcol,haggkvistC9} and \cite[Chapter 16]{bang2009}.
Let $G=(X,Y,E)$ be a bipartite graph for which each edge is coloured red or blue. Let $D=D(G)=(X,Y,A)$ be the bipartite digraph that we obtain from  $G$ by orienting every red edge $xy$, $x\in X,y\in Y$, as the arc $x\dom y$ and every blue edge $x'y'$,$x'\in X,y'\in Y$,  as the arc 
$y'\dom x'$. Now every alternating path, cycle, trail or walk in $G$ corresponds to a directed path, cycle, trail or walk  in $D$. It is clear that we  can also go the other way by replacing each arc from $X$ to $Y$ by a red edge and each other arc by a blue edge. 
This is called the {\bf BB-correspondence} in \cite[Chapter 16]{bang2009}.

The following is an immediate consequence of the BB-correspondence and well-known fact that the hamiltonian cycle problem is NP-complete for strongly connected bipartite digraphs.

\begin{theorem}\label{thm:npcbip2ec}
It is NP-complete to decide whether a colour-connected 2-edge-coloured bipartite graph has an alternating hamiltonian cycle.
\end{theorem}

Using the BB-correspondence we can now characterize  2-edge-coloured complete bipartite graphs with a hamiltonian cycle. This is not new (see e.g. \cite[Theorem 16.7.4]{bang2009}) and we only include it to illustrate the usefulness of the BB-correspondence.
A bipartite tournament is a bipartite digraph with partition classes $X$ and $Y$ such that there is precisely on arc 
between each  vertex of $X$ and each vertex of $Y$. By the BB-correspondence, each 2-edge-coloured complete bipartite graph $G$ corresponds to a bipartite tournament $B(G)$ and it is easy to see that $G$ is colour-connected if and only if $B(G)$ is strongly connected.  
It was shown in \cite{gutinVANB1984,haggkvistC9} that a bipartite tournament has a directed hamiltonian cycle if and only if it is strong and contains a factor, By the BB-correspondence this directly translates into the following characterization of 2-edge-coloured complete bipartite graphs with an alternating hamiltonian cycle.

\begin{theorem}\label{thm:bip2echam}
A 2-edge-coloured complete bipartite graph has an alternating hamiltonian cycle if and only if it is colour-connected and has an alternating cycle factor.
\end{theorem}

Saad \cite{saadCPC5} proved the following characterization of the length of a longest alternating cycle in a colour-connected 2-edge-coloured complete graph.

\begin{theorem}\cite{saadCPC5}
\label{thm:saad}
Let $G$ be a colour-connected  2-edge-coloured complete graph. The length of a longest alternating cycle in $G$ is equal to the maximum number of vertices that can be covered by disjoint alternating cycles in $G$.
\end{theorem}

Theorem \ref{thm:saad} immediately implies the following result due to Bankfalvi and Bankfalvi who formulated it in a different, but equivalent way.

\begin{theorem}\cite{bankfalvi1968}
    \label{thm:HCchar}
   Let $H$ be a 2-edge-coloured complete graph. Then $H$ has an alternating hamiltonian cycle
if and only if $H$ is colour-connected and has an alternating cycle factor. 
  \end{theorem}

Bang-Jensen and Gutin generalized this to extensions of 2-edge-coloured complete graph $G$.

\begin{theorem}\cite{bangDM188}
    \label{thm:extHCchar}
   Let $H$ be an extended 2-edge-coloured complete graph $G$. Then $H$ has an alternating hamiltonian cycle
if and only if $H$ is colour-connected and has an alternating cycle factor. 
  \end{theorem}

\jbj{
\section{Trail-colour-connectivity}

\jbjb{
We call a 2-edge-coloured graph $G=(V,E)$ {\bf trail-colour-connected} if $G$ contains two alternating $(u,v)$-trails  $T_1,T_2$ whose union is an alternating walk for every pair distinct vertices $u,v$.
The following analogous of Lemma \ref{lem:ccsuff} is easy to derive using almost the same proof as that of Lemma \ref{lem:ccsuff}.

 \begin{lemma}\label{lem:ccsuffTRAIL}
 Let $G$ be a 2-edge-coloured graph. Then $G$ is trail-colour-connected if and only if $G$ has  an alternating  $(u,v)$-trail starting with colour $c$ for each colour $c\in\{1,2\}$ and  every ordered pair of vertices $u,v$.
   \end{lemma}}

\jbjb{It is not difficult to prove the following extension of Proposition \ref{prop:extcc}}.

\begin{prop}
  \label{prop:Tccext}
  For a 2-edge-coloured graph $G$ the following are equvalent.
  \begin{itemize}
  \item[(i)] $G$ is trail-colour-connected.
    \item[(ii)] Every extension  $H$ of $G$ is trail-colour-connected.
    \end{itemize}
  \end{prop}

  \jbjb{The graph in Figure \ref{fig:needall} shows that we cannot replace 'every' by 'some' in Proposition \ref{prop:Tccext}(ii).}
  
\begin{figure}[H]
\begin{center}
\begin{minipage}{4cm}
\tikzstyle{vertexB}=[circle,draw, minimum size=12pt, scale=0.7, inner sep=0.5pt]
\tikzstyle{vertexBsmall}=[circle,draw, minimum size=10pt,  top color=gray!50, bottom color=gray!10, scale=0.8, inner sep=0.5pt]
\begin{tikzpicture}[scale=0.6]

 \node (x1) at (1,1) [vertexB] {$x_1$};
 \node (x2) at (1,5) [vertexB] {$x_2$};
 \node (u) at (3,3) [vertexB] {$u$};
 \node (v) at (5,3) [vertexB] {$v$};
 \node (y1) at (7,1) [vertexB] {$y_1$};
 \node (y2) at (7,5) [vertexB] {$y_2$};

  \draw [color=blue] (x1) to (x2);
  \draw [color=blue] (y1) to (y2);
  \draw [color=blue] (u) to (v);
  \draw [color=red] (x1) to (u);
  \draw [color=red] (x2) to (u);
  \draw [color=red] (y2) to (v);
  \draw [color=red] (y1) to (v);
  \node [scale=0.99] at (4,0.5) {$G$};
\end{tikzpicture}
\end{minipage} \hspace{1cm}
\begin{minipage}{4cm}
\tikzstyle{vertexB}=[circle,draw, minimum size=12pt, scale=0.7, inner sep=0.5pt]
\tikzstyle{vertexBsmall}=[circle,draw, minimum size=10pt,  top color=gray!50, bottom color=gray!10, scale=0.8, inner sep=0.5pt]
\begin{tikzpicture}[scale=0.6]


 \node (x1) at (1,1) [vertexB] {$x_1$};
 \node (x2) at (1,5) [vertexB] {$x_2$};
 \node (u1) at (3,2.4) [vertexB] {$u_1$};
 \node (u2) at (3,3.6) [vertexB] {$u_2$};
 \node (v1) at (5,2.4) [vertexB] {$v_1$};
 \node (v2) at (5,3.6) [vertexB] {$v_2$};
 \node (y1) at (7,1) [vertexB] {$y_1$};
 \node (y2) at (7,5) [vertexB] {$y_2$};

  \draw [color=blue] (x1) to (x2);
  \draw [color=blue] (y1) to (y2);

  \draw [color=blue] (u1) to (v1);
  \draw [color=blue] (u1) to (v2);
  \draw [color=blue] (u2) to (v1);
  \draw [color=blue] (u2) to (v2);

  \draw [color=red] (x1) to (u1);
  \draw [color=red] (x1) to (u2);
  \draw [color=red] (x2) to (u1);
  \draw [color=red] (x2) to (u2);
  \draw [color=red] (y2) to (v1);
  \draw [color=red] (y2) to (v2);
  \draw [color=red] (y1) to (v1);
  \draw [color=red] (y1) to (v2);

  \node [scale=0.99] at (4,0.5) {$H$};
\end{tikzpicture}
\end{minipage} 
\end{center}

\caption{$G$ is not trail-colour-connected as there is no trail from $x_1$ to $x_2$ starting with the color red.
$H$ is an extension of $G$ and $H$ is trail-colour-connected (and colour-connected) as it contains an alternating hamilton cycle 
$y_1 v_1 u_1 x_1 x_2 u_2 v_2 y_2 y_1$.}\label{fig:needall}
\end{figure}
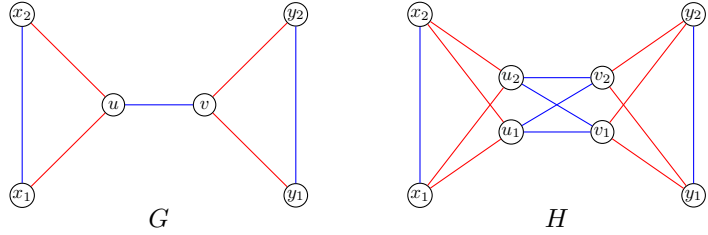
\vspace{0.5cm}

\begin{lemma}
\label{lem:CMGtcciffcc}
A 2-edge-coloured complete multipartite graph is colour-connected if and only if it is trail-colour-connected.
\end{lemma}

\begin{proof}
Let $G=(V,E)$ be a complete multipartite graph, let $\phi: E\rightarrow \{1,2\}$ be a 2-colouring of its edges and let $u,v$ be arbitrary distinct vertices. Assume first that $G$ is colour-connected. Then $G$ has a $(u,v)$-path $P_i$ starting on colour $i$ for $i\in [2]$. Hence we can take $P_1, P_2$ as the desired trails.

Suppose now that $G$ has alternating $(u,v)$-trails $T_1,T_2$ such that $T_i$ starts with an edge of colour $i$ for $i\in [2]$. We want to show that $G$ also has alternating $(u,v)$-paths $P_1,P_2$ such that  $P_i$ starts with an edge of colour $i$ for $i\in [2]$. Consider  the trail $T_c$. If it is a path, then we can take $P_c=T_c$ so assume that $T_c$ is not a path.
Below we show that, starting from the trail $T'=T_c$, if the current $(u,v)$-trail $T'$ starting with colour $c$ is not a path, then we can obtain a shorter $(u,v)$-trail that also starts on colour $c$ and hence by setting $T'$ to be this trail and iterating the procedure, we will  obtain the desired path $P_c$.\\

 If no vertex appears twice in $T'$, then, we can take $P_c=T'$.  Otherwise, let $w$ be the first vertex that is met at least twice as we traverse $T'$ and consider two possibilities.
\begin{itemize}
\item If there is an even number of edges between the two first occurrences of $w$, then we can remove the subtrail between these two occurrences and still have a properly-coloured $(u,v)$-trail $T^*$ which is shorter than $T'$. 
\item If there is an odd number of edges between the two first occurrences of $w$, then let $x$ be the first vertex that appears on $T'$ just after the first occurrence of $w$. Hence, we have two $(w,x)$-paths (one is just the edge $wx$) that start with the same colour and end with different colours.  Both of them can be used after the path $T'[u,w]$ that $T'$ defines between $u$ and $w$ and one of them is compatible with the edge between $x$ and $v$ in $G$ if such an edge exists so if $x$ and $v$ are adjacent we have found the desired path $P_c$. Hence we may asumme that  $x$ and $v$ are not adjacent, then $w$ and $v$ are adjacent (as $G$ is complete multipartite). If $\phi{}(wx)=\phi{}(wv)$  we can take $P=T'[u,w]wv$, so assume that $\phi{}(wx)\neq \phi{}(wv)$. In that case, either $T'[u,w]wv$ is not alternating as the last two edges have the same colour, or $u=w$ and $\phi{}(uv)\neq c$. Consider instead the predecessor $x^-$ of $x$ on $T'$. As $x$ and $v$ are non-adjacent, there is an edge between $x^-$ and $v$. If $\phi{}(x^-v)=\phi{}(wx)$, then $T'[u,w]wxx^-v$ is an alternating $(u,v)$-path and otherwise $T'[u,x^-]x^-v$ is an alternating $(u,v)$-path.
\end{itemize}

As the two cases above cover all possibilities, we see that, by iterating the process above, we will eventually end up with an alternating  $(u,v)$-path which starts with colour $c$. Since $c$ was arbitrary, we have shown that the $(u,v)$-paths $P_1,P_2$ exist for every choice of distinct vertices $u,v$. Hence, by Lemma \ref{lem:ccsuff}, $G$ is colour-connected.

The procedure is illustrated in Figure \ref{figcolcon}. For readability, only the edges of the spanning eulerian subgraph are represented in the figure, but keep in mind that the graph is complete multipartite graph. Let us look for a path from $v_1$ to $v_9$ starting with a red edge. The eulerian subgraph defines a walk $v_1v_2v_3v_4v_5v_6v_3v_7v_2v_8v_9$ starting with a red edge but one cannot extract a path from it. The first vertex that appears twice in the walk in $v_3$. Since the walk uses an even number of edges between the two first occurrences of $v_3$, we can contract those edges and find the properly-coloured walk $v_1v_2v_3v_7v_2v_8v_9$. The first vertex that appears twice is $w=v_2$ and the walk uses an odd number of edges between the first two occurrences. Hence, the walk defines two paths $v_2v_3$ and $v_2v_7v_3$ between $v_2$ and $v_3$. If there is a red edge between $v_3$ and $v_9$, our path from $v_1$ to $v_9$ starting with a red edge is $v_1v_2v_3v_9$. If there is a blue edge $v_3$ between and $v_9$, the path is $v_1v_2v_7v_3v_9$. Finally, if there is no edge between $v_3$ and $v_9$, then there must be a blue edge between $v_2$ and $v_9$ since there is one between $v_2$ and $v_3$. Hence, the path would be $v_1v_2v_9$.\\

\begin{figure}[H]

\centering{\includegraphics{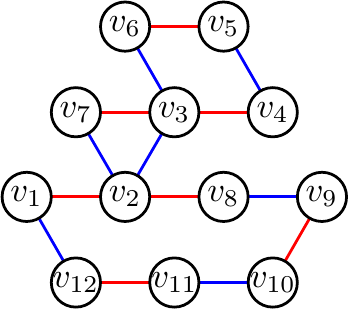}}
\caption{A spanning eulerian subgraph in a 2-edge-coloured graph.}
\label{figcolcon}
\end{figure}

\end{proof}

}

The following lemma can be found \cite{bangDM188} (as Proposition 5.1) and also in
\cite{bang2009} (as Proposition 16.6.2).

\begin{lemma} \label{lemPolPath} \cite{bangDM188,bang2009}
Let $G=(V,E)$ be a connected 2-edge-coloured multigraph and let $x$ and $y$ be distinct vertices of $G$.
For each choice of $i,j \in \{1,2\}$ we can find an alternating path $P=x_1 x_2 \ldots x_k$ with
$x_1=x$, $x_k=y$, $\phi{}(x_1x_2)=i$ and $\phi{}(x_{k-1}x_k)=j$ in time $O(|E|)$ (if one exists).
\end{lemma}

\begin{theorem} \label{thmTRAIL}
Let $G$ be a $2$-edge coloured graph and let $x,y \in V(G)$ be arbitrary.
We can decide if there is a trail from $x$ to $y$ starting with colour $c_1$ and ending with colour $c_2$ in polynomial time.
\end{theorem}

\begin{proof}
Let $G$ be a $2$-edge coloured graph and let $x,y \in V(G)$ be arbitrary.
Build a new $2$-edge coloured graph $H$ as follows (see Figure~\ref{TrailPic} for an example). Let $V(H)$ be defined as follows.

\[
V(H)  =  \{ u_1, u_2 \; | \; u \in V(G) \}  \cup \{ u_{uv}, v_{uv} \; | \; uv \in E(G) \}\\
\]

For every edge $uv \in E(G)$, add the following edges
$u_1 u_{uv}, u_2 u_{uv}, v_1 v_{uv}, v_2 v_{uv}$
to $H$ with colour $\phi{}(uv)$ and add the edge $u_{uv} v_{uv}$ to $H$ with colour $3-\phi{}(uv)$ (see the picture below for an example of $H$).

We will now show that there is a $(x,y)$-alternating-trail in $G$ starting with colour $c_1$ and
ending with colour $c_2$  if and only if there is a 
$(x_1,y_1)$-alternating-path in $H$ starting with colour $c_1$ and 
ending with colour $c_2$. This would complete the proof by Lemma~\ref{lemPolPath}

Let $P$ be an $(x_1,y_1)$-alternating-path in $H$ starting with colour $c_1$ and
ending with colour $c_2$.  
By substituting every subpath $u_j u_{uv} v_{uv} v_i$ of $P$ in $H$ by the edge $uv$ in $G$ we 
obtain a walk, $W$, in $G$. This walk is alternating and no edge is used more than once (as edges
of the type $u_{uv} v_{uv}$ are only used once in $P$).  Therefore we note that $W$ is a $(x,y)$-alternating-trail
in $G$  starting with colour $c_1$ and
ending with colour $c_2$.

Conversely let $T=t_1t_2\ldots t_k$ be a $(x,y)$-alternating-trail in $G$  starting with colour $c_1$ and ending with colour $c_2$. 
Assume furthermore that $T$ contains as few edges as possible. 
If there are two edge $t_{i-1}t_{i}$ and $t_{j-1}t_j$, where $i<j$, of the same
colour in $W$ such that $t_i=t_j$ then we obtain a contradiction to the minimality of $T$ as we could have considered the trail
$t_1 t_2 \ldots t_{i-1} t_j t_{j+1} \ldots t_k$.
This implies that no vertex appears more than twice in $T$.
First substitute the $i$'th appearance of $u$ in $T$ by $u_i$ ($i \in \{1,2\}$) and then replace every 
edge $u_i v_j$ on $T$ by the path $u_i u_{uv} v_{uv} v_j$ ($i,j \in \{1,2\}$). This gives us a $(x_1,y_j)$-alternating-path
in $H$ (for some $j \in \{1,2\}$)  starting with colour $c_1$ and
ending with colour $c_2$. If $j=2$ then swap $y_1$ and $y_2$ in the path. 

This implies that
 there is a $(x,y)$-alternating-trail in $G$ starting with colour $c_1$ and
ending with colour $c_2$  if and only if there is a
$(x_1,y_1)$-alternating-path in $H$ starting with colour $c_1$ and
ending with colour $c_2$, as desired.
\end{proof}


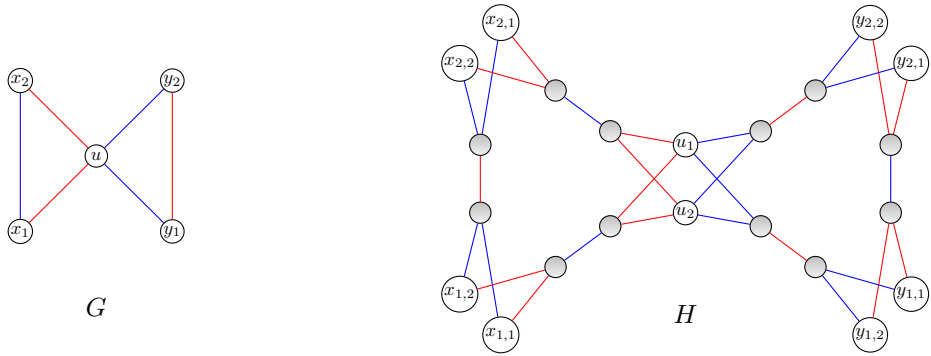
\begin{figure}[H]
\tikzstyle{vertexB}=[circle,draw, minimum size=12pt, scale=0.7, inner sep=0.5pt]
\tikzstyle{vertexBsmall}=[circle,draw, minimum size=10pt,  top color=gray!50, bottom color=gray!10, scale=0.8, inner sep=0.5pt]
\begin{tikzpicture}[scale=0.5]
 \node (x1) at (1,1) [vertexB] {$x_1$};
 \node (x2) at (1,5) [vertexB] {$x_2$};
 \node (u) at (3,3) [vertexB] {$u$};
 \node (y1) at (5,1) [vertexB] {$y_1$};
 \node (y2) at (5,5) [vertexB] {$y_2$};

  \draw [color=blue] (x1) to (x2);
  \draw [color=red] (x1) to (u);
  \draw [color=red] (x2) to (u);
  \draw [color=blue] (y2) to (u);
  \draw [color=blue] (y1) to (u);
  \draw [color=red] (y1) to (y2);
  \node [scale=0.99] at (3,-1) {$G$};
  \node [scale=0.99] at (3,-2) { };

\end{tikzpicture} \hfill
\begin{tikzpicture}[scale=0.9]
 \node [scale=0.99] at (4,1) {$H$};

 \node (x11) at (1.3,0.7) [vertexB] {$x_{1,1}$};
 \node (x12) at (0.7,1.3) [vertexB] {$x_{1,2}$};

 \node (x21) at (1.3,5.3) [vertexB] {$x_{2,1}$};
 \node (x22) at (0.7,4.7) [vertexB] {$x_{2,2}$};

 \node (u1) at (4,3.5) [vertexB] {$u_1$};
 \node (u2) at (4,2.5) [vertexB] {$u_2$};


 \node (y11) at (7.3,1.3) [vertexB] {$y_{1,1}$};
 \node (y12) at (6.7,0.7) [vertexB] {$y_{1,2}$};

 \node (y21) at (7.3,4.7) [vertexB] {$y_{2,1}$};
 \node (y22) at (6.7,5.3) [vertexB] {$y_{2,2}$};

 \node (x1x2a) at (1,2.5) [vertexBsmall] {};
 \node (x1x2b) at (1,3.5) [vertexBsmall] {};
  \draw [color=blue] (x11) to (x1x2a);
  \draw [color=blue] (x12) to (x1x2a);
  \draw [color=blue] (x21) to (x1x2b);
  \draw [color=blue] (x22) to (x1x2b);
  \draw [color=red] (x1x2a) to (x1x2b);

 \node (y1y2a) at (7,2.5) [vertexBsmall] {};
 \node (y1y2b) at (7,3.5) [vertexBsmall] {};
  \draw [color=red] (y11) to (y1y2a);
  \draw [color=red] (y12) to (y1y2a);
  \draw [color=red] (y21) to (y1y2b);
  \draw [color=red] (y22) to (y1y2b);
  \draw [color=blue] (y1y2a) to (y1y2b);

 \node (x2ua) at (2.1,4.3) [vertexBsmall] {};
 \node (x2ub) at (2.9,3.7) [vertexBsmall] {};
  \draw [color=red] (x21) to (x2ua);
  \draw [color=red] (x22) to (x2ua);
  \draw [color=red] (u1) to (x2ub);
  \draw [color=red] (u2) to (x2ub);
  \draw [color=blue] (x2ua) to (x2ub);

 \node (y2ua) at (5.9,4.3) [vertexBsmall] {};
 \node (y2ub) at (5.1,3.7) [vertexBsmall] {};
  \draw [color=blue] (y21) to (y2ua);
  \draw [color=blue] (y22) to (y2ua);
  \draw [color=blue] (u1) to (y2ub);
  \draw [color=blue] (u2) to (y2ub);
  \draw [color=red] (y2ua) to (y2ub);

 \node (x1ua) at (2.1,1.7) [vertexBsmall] {};
 \node (x1ub) at (2.9,2.3) [vertexBsmall] {};
  \draw [color=red] (x11) to (x1ua);
  \draw [color=red] (x12) to (x1ua);
  \draw [color=red] (u1) to (x1ub);
  \draw [color=red] (u2) to (x1ub);
  \draw [color=blue] (x1ua) to (x1ub);

 \node (y1ua) at (5.9,1.7) [vertexBsmall] {};
 \node (y1ub) at (5.1,2.3) [vertexBsmall] {};
  \draw [color=blue] (y11) to (y1ua);
  \draw [color=blue] (y12) to (y1ua);
  \draw [color=blue] (u1) to (y1ub);
  \draw [color=blue] (u2) to (y1ub);
  \draw [color=red] (y1ua) to (y1ub);
\end{tikzpicture}
\caption{An example of the transformation of $G$ into $H$ used in the proof of Theorem~\ref{thmTRAIL}.
Note that any alternating $(a,b)$-path in $H$ ($a,b \in \{x_{1,1},x_{1,2},x_{2,1},x_{2,2},u_1,u_2,y_{1,1},y_{1,2},y_{2,1},y_{2,2}\}$)
corresponds to an alternating $(a,b)$-trail in $G$ and vica versa.}
\label{TrailPic}
\end{figure}

\begin{corollary}
We can decide if a $2$-edge coloured graph is trail-colour-connected in polynomial time.
\end{corollary}

\section{Eulerian factors and supereulerian graphs}\label{sec:supereuler}
Recall that a connected  undirected graph is eulerian if it has a spanning closed trail which uses every edge. By Euler's theorem \cite{eulerCASP8}, $G$ is eulerian if and only if it is connected and the degree of every vertex is even. This can be generalized to 2-edge-coloured graphs as follows.
 A 2-edge coloured graph $F$ is {\bf eulerian} if it contains a closed alternating trail which covers all the edges of $G$. Following the standard proof of Euler's theorem is easy to see that a connected 2-edge coloured graph $G$ is eulerian if and only if each vertex $v$ has even degree and half of the edges incident to $v$ have colour $i$ for $i\in [2]$. For a more general result on properly coloured eulertours in $k$-edge-coloured graphs, see \cite{kotzigMFC18}.  Following the same definitions for graphs and digraphs, we say that a  2-edge-coloured graph $G$ is {\bf supereulerian} if it contains a spanning subgraph which is eulerian. This is equivalent to saying that $G$ contains a spanning closed alternating trail.

An {\bf eulerian factor} of a 2-edge-coloured graph $G$ is a collection of 
vertex-disjoint induced subgraphs $G_1=(V_1,E_1),\ldots{},G_k=(V_k,E_k)$  of $G$, such that $V=V_1\cup\ldots\cup{}V_k$ and each $G_i$ is supereulerian.

\begin{lemma}\label{lem:findEF}
There exists a polynomial algorithm  for finding an eulerian factor of a 2-edge-coloured graph $G$ or producing a certificate that $G$ has no such factor.
\end{lemma}

\begin{proof}
Let $G$ be a 2-edge-coloured graph.
 We will construct a new graph, $H$, such that $H$ has a perfect matching if and only if $G$ has a eulerian factor.
For each vertex $x \in V(G)$ do the following.  Let $x$ be incident with $b(x)$ blue edges and $r(x)$ red edges. Let $R(x),R’(x),B(x),B’(x)$ be vertex-sets in $H$ of the following sizes $|R(x)|=r(x)=|R'(x)|+1$ and $|B(x)|=b(x)=|B’(x)|+1$.  Now add all edges between $R(x)$ and $R’(x)$ and between $R’(x)$ and $B’(x)$ and between $B’(x)$ and $B(x)$.  Finally add an edge from $B(x)$ to $B(y)$ if there is a blue edge $xy$ in $G$ and add an edge between $R(x)$ and $R(y)$ if there is a red edge in $G$. And do this such that each vertex in $R(x)$ and $B(x)$ is incident with exactly one such edge (which can be done by the construction of $R(x),B(x)$).
  Now it is easy to check that $H$ has a perfect matching if and only if $G$ has a eulerian factor.  The factor goes through $x$ $k$ times if the matching uses $k-1$ edges between $B’(x)$ and $R’(x)$. See Figure \ref{fig:Efig} and \ref{fig:Mfig} for an illustration of the reduction.
\end{proof}

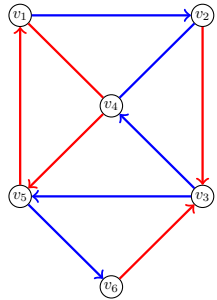
\begin{figure}[H]
\begin{center}
\tikzstyle{vertexr}=[circle,draw, color=red, minimum size=7pt, scale=0.6, inner sep=0.5pt]
\tikzstyle{vertexb}=[circle,draw, color=blue,minimum size=7pt, scale=0.6, inner sep=0.5pt]
\tikzstyle{vertexN}=[circle,draw, minimum size=14pt, scale=0.6, inner sep=0.5pt]
\tikzstyle{vertexR}=[circle,draw, top color=red!100, bottom color=red!100, minimum size=7pt, scale=0.6, inner sep=0.5pt]
\tikzstyle{vertexB}=[circle,draw, top color=blue!100, bottom color=blue!100, minimum size=7pt, scale=0.6, inner sep=0.5pt]
\begin{tikzpicture}[scale=0.6]
\node (y1) at (0,6) [vertexN]{$v_1$};
\node (y2) at (4,6) [vertexN]{$v_2$};
\node (y3) at (4,2) [vertexN]{$v_3$};
\node (y4) at (2,4) [vertexN]{$v_4$};
\node (y5) at (0,2) [vertexN]{$v_5$};
\node (y6) at (2,0) [vertexN]{$v_6$};
\draw[->,  line width=0.03cm, color=blue]  (y1) to (y2);
\draw[->,  line width=0.03cm, color=red]  (y2) to (y3);
\draw[->,  line width=0.03cm, color=blue]  (y3) to (y4);
\draw[->,  line width=0.03cm, color=red]  (y4) to (y5);
\draw[->,  line width=0.03cm, color=blue]  (y5) to (y6);
\draw[->,  line width=0.03cm, color=red]  (y6) to (y3);
\draw[->,  line width=0.03cm, color=blue]  (y3) to (y5);
\draw[->,  line width=0.03cm, color=red]  (y5) to (y1);
\draw[line width=0.03cm, color=red]  (y4) to (y1);
\draw[line width=0.03cm, color=blue]  (y4) to (y2);

\end{tikzpicture}
\caption{A 2-edge-coloured graph $G$ with a spanning closed alternating trail in $G$ (indicated as directed edges).}\label{fig:Efig}
\end{center}
\end{figure}
\newcommand{\ls}{line width=0.03cm, dotted}
\newcommand{\lf}{line thick}
\newcommand{\lsb}{line width=0.03cm, color=blue}
\newcommand{\lsr}{line width=0.03cm, color=red}
\newcommand{\lB}{line thick, color=blue}
\newcommand{\lR}{line thick, color=red}

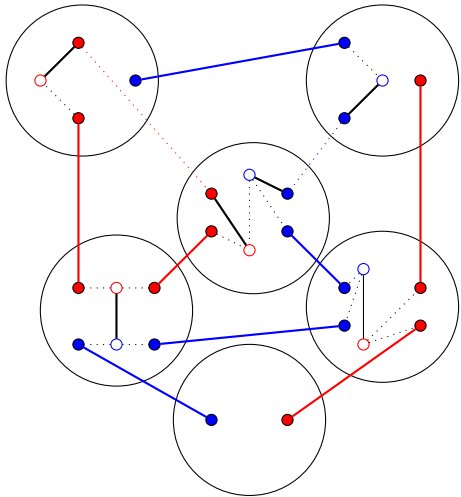
\begin{figure}[H]
\begin{center}
\tikzstyle{vertexr}=[circle,draw, color=red, minimum size=7pt, scale=0.6, inner sep=0.5pt]
\tikzstyle{vertexb}=[circle,draw, color=blue,minimum size=7pt, scale=0.6, inner sep=0.5pt]
\tikzstyle{vertexN}=[circle,draw, minimum size=14pt, scale=0.6, inner sep=0.5pt]
\tikzstyle{vertexR}=[circle,draw, top color=red!100, bottom color=red!100, minimum size=7pt, scale=0.6, inner sep=0.5pt]
\tikzstyle{vertexB}=[circle,draw, top color=blue!100, bottom color=blue!100, minimum size=7pt, scale=0.6, inner sep=0.5pt]
\tikzstyle{vertexC}=[circle,draw,  minimum size=95pt, scale=0.6, inner sep=0.5pt]

\begin{tikzpicture}[scale=0.5]
\node () at (1.1,9) [vertexC] {};
\node () at (9,9) [vertexC] {};
\node () at (9,3) [vertexC] {};
\node () at (5.6,5.35) [vertexC] {};
\node () at (2,2.9) [vertexC] {};
\node () at (5.5,0) [vertexC] {};
\node (y11) at (0,9) [vertexr] {};
\node (y12) at (1,10) [vertexR] {};
\node (y13) at (1,8) [vertexR] {};
\node (y14) at (2.5,9) [vertexB] {};
\node (y21) at (8,10) [vertexB] {};
\node (y22) at (8,8) [vertexB] {};
\node (y23) at (9,9) [vertexb] {};
\node (y24) at (10,9) [vertexR] {};
\node (y41) at (4.5,6) [vertexR] {};
\node (y42) at (4.5,5) [vertexR] {};
\node (y43) at (5.5,4.5) [vertexr] {};
\node (y44) at (5.5,6.5) [vertexb] {};
\node (y45) at (6.5,6) [vertexB] {};
\node (y46) at (6.5,5) [vertexB] {};
\node (y31) at (8,3.5) [vertexB] {};
\node (y32) at (8,2.5) [vertexB] {};
\node (y35) at (10,3.5) [vertexR] {};
\node (y36) at (10,2.5) [vertexR] {};
\node (y33) at (8.5,4) [vertexb] {};
\node (y34) at (8.5,2) [vertexr] {};
\node (y51) at (1,3.5) [vertexR] {};
\node (y52) at (3,3.5) [vertexR] {};
\node (y53) at (2,3.5) [vertexr] {};
\node (y54) at (1,2) [vertexB] {};
\node (y55) at (3,2) [vertexB] {};
\node (y56) at (2,2) [vertexb] {};
\node (y61) at (4.5,0) [vertexB] {};
\node (y62) at (6.5,0) [vertexR] {};
\draw[thick, color=blue] (y54) to (y61);

\draw[dotted] (y13) to (y11);
\draw[thick] (y11) to (y12);
\draw[thick, color=blue] (y14) to (y21);
\draw[dotted, color=red] (y12) to (y41);
\draw[thick, color=red] (y13) to (y51);
\draw[dotted] (y21) to (y23);
\draw[thick] (y22) to (y23);
\draw[dotted, color=blue] (y22) to (y45);
\draw[thick, color=red] (y24) to (y35);
\draw[dotted] (y33) to (y31);
\draw[dotted] (y33) to (y32);
\draw (y33) to (y34);
\draw[dotted] (y34) to (y35);
\draw[dotted] (y34) to (y36);
\draw[thick, color=blue] (y31) to (y46);
\draw[thick, color=blue] (y32) to (y55);
\draw[thick, color=red] (y36) to (y62);
\draw[dotted] (y42) to (y43);
\draw[thick] (y41) to (y43);
\draw[thick] (y44) to (y45);
\draw[dotted] (y44) to (y43);
\draw[dotted] (y44) to (y46);
\draw[thick, color=red] (y42) to (y52);
\draw[dotted] (y53) to (y51);
\draw[dotted] (y53) to (y52);
\draw[thick] (y53) to (y56);
\draw[dotted] (y56) to (y54);
\draw[dotted] (y56) to (y55);

\end{tikzpicture}
\caption{The graph $H=H(G)$ constructed as in the proof of Lemma \ref{lem:findEF}. The perfect matching corresponding to the spanning eulerian subgraph indicated in Figure \ref{fig:Efig} is shown with full lines. The colours are just for easy reference to the other figure.}\label{fig:Mfig}
\end{center}
\end{figure}

The following result on supereulerian digraphs is Theorem 2.8 in \cite{bangJGT79}. A digraph is {\bf semicomplete multipartite} if the underlying undirected graph is a complete multipartite graph. 
\begin{theorem}\cite{bangJGT79}
A semicomplete multipartite digraph is supereulerian if and only if is is strongly connected and has an eulerian factor.
\end{theorem}

Since a bipartite tournament is a semicomplete multipartite digraph, the BB-correspondence implies the following characterization of supereulerian 2-edge-coloured complete bipartite graphs.

\begin{corollary}
\label{cor:compbip}
A 2-edge-coloured complete bipartite graph $G$ is supereulerian if and only if $G$ is colour-connected and has an eulerian factor.
\end{corollary}

As both the problem of deciding if a 2-edge-coloured graph is colour-connected  and the problem of deciding if it contains an eulerian factor are polynomial
time solvable, we note that Corollary~\ref{cor:compbip} implies that we in polynomial time can decide if a 
2-edge-coloured complete bipartite graph is supereulerian.

\section{Alternating Hamiltonian cycles in extensions of M-closed 2-edge-coloured graphs}
In \cite{balbuenaDMTCS21} the authors consider a generalization of 2-edge-coloured complete multigraphs, namely those 2-edge-coloured graphs for which the end-vertices of every monochromatic path of length 2 are adjacent, that is, if $xyz$ is a path and $\phi{}(xy)=\phi{}(yz)$, then $xz$ is an edge of the graph. The authors call such graphs {\bf M-closed}. The following is the main result of \cite{balbuenaDMTCS21}.

\begin{theorem}\cite{balbuenaDMTCS21}
  \label{thm:MclHC}
Let $G$ be a 2-edge-coloured graph which is M-closed. Then $G$ has an alternating hamiltonian cycle if and only if it is colour-connected and has an
alternating cycle factor. 
\end{theorem}

The example in Figure \ref{fig:halfM}, which can easily be extended to an infinite family, shows that the definition of being M-closed cannot be relaxed to a requirement only for monochromatic paths of length 2 of one of the two colours.

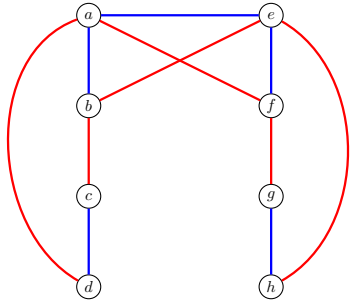
\begin{figure}[H]
\begin{center}
\tikzstyle{vertexr}=[circle,draw, color=red, minimum size=7pt, scale=0.6, inner sep=0.5pt]
\tikzstyle{vertexb}=[circle,draw, color=blue,minimum size=7pt, scale=0.6, inner sep=0.5pt]
\tikzstyle{vertexN}=[circle,draw, minimum size=14pt, scale=0.6, inner sep=0.5pt]
\tikzstyle{vertexR}=[circle,draw, top color=red!100, bottom color=red!100, minimum size=7pt, scale=0.6, inner sep=0.5pt]
\tikzstyle{vertexB}=[circle,draw, minimum size=15pt, scale=0.6, inner sep=0.5pt]
\begin{tikzpicture}[scale=0.6]
\node(a) at (0,6) [vertexB] {$a$};
\node(b) at (0,4) [vertexB] {$b$};
\node(c) at (0,2) [vertexB] {$c$};
\node(d) at (0,0) [vertexB] {$d$};
\node(e) at (4,6) [vertexB] {$e$};
\node(f) at (4,4) [vertexB] {$f$};
\node(g) at (4,2) [vertexB] {$g$};
\node(h) at (4,0) [vertexB] {$h$};
\draw[line width=0.03cm, color=blue] (a) to (b);
\draw[line width=0.03cm, color=blue] (c) to (d);
\draw[line width=0.03cm, color=blue] (e) to (f);
\draw[line width=0.03cm, color=blue] (g) to (h);
\draw[line width=0.03cm, color=red] (b) to (c);
\draw[line width=0.03cm, color=red] (d) to [out=150,in=200] (a);
\draw[line width=0.03cm, color=red] (f) to (g);
\draw[line width=0.03cm, color=red] (e) to [out=-30,in=30] (h);
\draw[line width=0.03cm, color=blue] (a) to (e);
\draw[line width=0.03cm, color=red] (a) to (f);
\draw[line width=0.03cm, color=red] (b) to (e);

\end{tikzpicture}
\caption{A 2-edge-coloured graph $G$ in which the end vertices $x,z$ are adjacent for  every path $xyz$ with $\phi{}(xy)=\phi{}(yz)=1$ (1=blue). $G$ is colour-connected and has a cycle factor but it has no alternating hamiltonian cycle. It also has no spanning closed alternating trail.}\label{fig:halfM}
\end{center}
\end{figure}

\begin{figure}[H]
\centering{\includegraphics{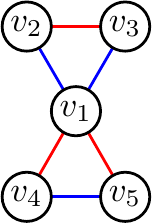}}
\caption{A non colour-connected graph with a spanning closed alternating trail.}
\label{fig:notcc}
\end{figure}
The 2-edge-coloured graph in Figure \ref{fig:notcc} is M-closed and is eulerian but not colour-connected. Hence for M-closed 2-edge-coloured graphs, having a spanning closed alternating trail does not imply colour-connectivity. \jbjb{Note that the graph is trail-colour-connected, as is every 2-edge-coloured graph with a spanning closed altenating trail.}\\

\jbj{
We will now argue that by carefully checking the proof of Theorem \ref{thm:MclHC}  in \cite{balbuenaDMTCS21}  one can verify that the following generalization holds. 

\begin{theorem}\label{thm:extMclHC}
  Let $G$ be an extension of an  M-closed 2-edge-coloured graph. Then $G$ has an alternating hamiltonian cycle if and only if $G$ is colour-connected and has an alternating cycle factor.
\end{theorem}

The proof of Theorem \ref{thm:MclHC}  is based on the following Lemmas. For each we will argue why they can be extended to extensions of M-closed 2-edge-coloured graphs.  Before we list the lemmas, we recall the following easy fact about pairs of alternating cycles in extended 2-edge-coloured graphs.

\begin{prop}\label{prop:similar}
  Let $C_1=x_1x_2\ldots{}x_{2k-1}x_{2k}x_1$ and $C_2=y_1y_2\ldots{},y_{2r-1}y_{2r}$ be disjoint alternating cycles in a 2-edge-coloured graph $G$. If there exist indices $i\in [2k],j\in [2r]$ such that $x_i$ and $y_j$ are similar, then $G$ contains an alternating cycle $C$ with $V(C)=V(C_1)\cup V(C_2)$.
\end{prop}
\begin{proof}Assume that  $x_i$ and $y_j$ are similar. By reversing the ordering of one of the cycles if necessary we can assume that $\phi(x_ix_{i+1})=\phi(y_jy_{j+1})$. Now the fact that $x_i$ and $y_j$ are similar implies that the edges $x_{i-1}y_j, y_{j-1}x_i$ exist and $\phi(x_{i-1}y_j)=\phi(y_{j-1}x_i)$. Hence $C_1[x_i,x_{i-1}]x_{i-1}y_jC_2[y_j,y_{j-1}]y_{j-1}x_i$ is the desired cycle.
\end{proof}

The first lemma below, which holds for general 2-edge-coloured graphs, is very simple and has been used in many papers.

\begin{lemma}\label{lem4cycle}
  Let $C_1=x_1x_2\ldots{}x_{2k-1}x_{2k}x_1$ and $C_2=y_1y_2\ldots{},y_{2r-1}y_{2r}$ be disjoint alternating cycles in a 2-edge-coloured graph $G$. If there exist indices $i\in [2k],j\in [2r]$ such that $G$ contains both  of the edges $x_iy_j$ and $x_{i+1}y_{j+1}$ and $\phi{}(x_iy_j)=\phi{}(x_ix_{i+1})=\phi{}(x_{i+1}y_{j+1})=\phi{}(y_jy_{j+1})$, then $G$ contains an alternating cycle $C$ with $V(C)=V(C_1)\cup V(C_2)$.
\end{lemma}

\begin{lemma}\label{lem:altcls}\cite[Lemma 6]{balbuenaDMTCS21}
  Let  $G$ be an M-closed  2-edge-coloured graph and let $C_1=x_1x_2\ldots{}x_{2k-1}x_{2k}x_1$ and $C_2=y_1y_2\ldots{},y_{2r-1}y_{2r}$ be disjoint alternating cycles in  $G$. Suppose that the edge
  $x_iy_j$  exists in $G$ and $\phi{}(x_iy_j)=\phi{}(x_ix_{i+1})=\phi{}(y_jy_{j+1})=c$. Then either $G$ contains an alternating cycle $C$ with $V(C)=V(C_1)\cup V(C_2)$ or the edge $x_{i+1}y_{j+1}$ exists and $\phi{}(x_{i+1}y_{j+1})\neq c$.
\end{lemma}

To see that Lemma \ref{lem:altcls} holds for extensions of M-closed  2-edge-coloured graphs we first observe that, by Proposition \ref{prop:similar}, we can assume there is no pair of similar vertices $x_a,y_b$. This implies that all the arguments in the proof of the lemma in \cite{balbuenaDMTCS21} that deal with possible edges between the two cycles carry over to extended M-closed 2-edge-coloured graphs. There are only three places where edges between non consecutive vertices of the same cycle are used in the argument. In one case this is an edge of the kind
$x_{i-1}x_{i+1}$ in another it is an edge of the kind $x_{i-2}x_{i+1}$ and in the final case it is the edge $x_{i+1}x_{i+3}$. In all three  cases it is possible that the edge is not in $G$, because $G$ is an extension of an M-closed 2-edge-coloured multigraph, but then Proposition \ref{prop:similar} and the colours of  edges already studied in the original proof in  \cite{balbuenaDMTCS21} easily leads to the desired conclusion that the edge $x_{i+1}y_{j+1}$ is in $G$ and either $G$ has a cycle $C$ with $V(C)=V(C_1)\cup V(C_2)$ or we have $\phi{}(x_{i+1}y_{j+1})\neq c$.

\begin{lemma}\cite[Corollary 7]{balbuenaDMTCS21}
  \label{lem:cor7}
  Let $G$ be an  M-closed  2-edge-coloured graph and let $C_1$ and $C_2$ be disjoint alternating cycles of $G$ such that there is at least one edge between $C_1$ and $C_2$. Then at least one of the following holds:
  \begin{itemize}
  \item[1.] $G$ contains an alternating cycle $C$ with $V(C)=V(C_1)\cup V(C_2)$.
  \item[2.] Every vertex of $C_1$ is adjacent to every vertex of $C_2$.
    \end{itemize}
  \end{lemma}

  Besides applying Lemma \ref{lem:altcls} the proof of Lemma \ref{lem:cor7} in \cite{balbuenaDMTCS21} uses only arguments based on pairs of edges between the two cycles or an edge of one cycle and an edge between the cycles, so by Proposition \ref{prop:similar}, the Lemma also holds for extensions of M-closed 2-edge-coloured graphs.\\

  The following lemma, which is implicitly stated and proved on pages 8-10 in \cite{balbuenaDMTCS21},
  is the key to the proof of Theorem \ref{thm:MclHC}. We state it for extended M-closed 2-edge-coloured graphs as the statement is slighly different.  The only difference is that in (ii) and (iii) there may be pairs of similar vertices in $\{x_1,x_3,\ldots, x_{2p-1}\}$ and also in $\{x_2,x_4,\ldots, x_{2p}\}$ so the subgraph
  $G[\{x_1,x_3,\ldots, x_{2p-1}\}]$ as well as the subgraph $G[\{x_2,x_4,\ldots, x_{2p}\}]$ does not have to be complete as it is the case for the corresponding Lemma for  M-closed graphs. The proof of the lemma
  is the same as for M-closed 2-edge-coloured graphs.
  
  \begin{lemma}
    \label{lem:dom}
  Let $C_1,C_2$ be disjoint alternating cycles in an extended M-closed 2-edge-coloured graph
  $G$ such that there is at least one edge between $C_1$ and $C_2$. If $D$ has no alternating cycle $C$ with $V(C)=V(C_1)\cup V(C_2)$, then every vertex of  $C_1$ is adjacent to every vertex of  $C_2$ and for some $i\in [2]$
  the  vertices of  $C_i$ can be labelled such that $C_i=x_1x_2\ldots{}x_{2p}x_1$ and the following holds.
  \begin{enumerate}
  \item[i)] all edges between $\{x_1,x_3,\ldots, x_{2p-1}\}$ and $V(C_2)$ have the same colour $c=\phi{}(x_1x_2)$ and all the edges between $\{x_2,x_4,\ldots, x_{2p}\}$ and $V(C_2)$ have colour $c'\neq c$.
  \item[(ii)] Every edge between two vertices $x_{2i+1},x_{2j+1}$ has colour $c$.
    \item[(iii)] Every edge between two vertices $x_{2a},x_{2b}$ has colour $c'$.
    
    \end{enumerate}
    We say that $C_1$ {\bf $\mathbf{c}$-dominates} $C_2$ and denote it by $C_1{\stackrel{c}{\rightarrow}}C_2$. 
    \end{lemma}

It is easy to check that if $C_1{\stackrel{c}{\rightarrow}}C_2$, then $G[V(C_1)\cup{}V(C_2)]$ is not trail-colour-connected and hence also not colour-connected.

  The proof of the non-trivial direction in Theorem \ref{thm:extMclHC} now proceeds as follows:
  Consider a cycle factor $C_1,C_2,\ldots{},C_k$ with the minimum number of cycles. If $k=1$ the proof is complete and otherwise, by considering only edges between cycles, one obtains the contradiction that $G$ is not colour-connected. This proof carries over verbatim to the case of extended M-closed 2-edge-coloured graphs.}\\

The proofs in \cite{balbuenaDMTCS21} are algorithmic so, from the arguments above we get the following.

\begin{corollary}
  The exists a polynomial algorithm ${\cal A}$ which, given  a graph $G$ which is an extension of an M-closed 2-edge-coloured graph such that $G$ is colour-connected and has a cycle factor, produces an alternating hamiltonian cycle of $G$.
\end{corollary}

\section{Supereulerian  extensions of M-closed 2-edge-coloured  graphs}

\jbjb{Armed with Theorem \ref{thm:extMclHC} we are now ready to characterize supereulerian extensions of M-closed 2-edge-coloured graphs. Note that, by the example in Figure \ref{fig:notcc}, a supereulerian M-closed 2-edge-coloured graph does not have to be colour-connected, but it must be trail-colour-connected. We first consider the case of colour-connected graphs.}

\begin{lemma}
  \label{lem:exMclccE}
  Let $G$ be an extension of an M-closed 2-edge coloured multigraph. If $G$ is colour-connected and has an eulerian factor, then $G$ is supereulerian.
\end{lemma}

\begin{proof}
  Let $G$ be an
     extension of an M-closed 2-edge-coloured graph which is colour-connected and let ${\cal G}=G_1,G_2,\ldots{},G_k$, $k\geq 1$ be an eulerian factor of $G$ which is chosen such that $k$ is minimum.
    If $k=1$ there is nothing to prove so suppose $k\geq 2$. Let $T_i$ be a spanning closed trail in $G_i$ for $i\in [k]$. Let $h(v)$ be the number of times the vertex $v\in V$ occurs in the spanning closed trail $T_i$ that it belongs to. Since the $G_i$'s are disjoint and spanning, each $v$ occurs in exactly one $T_i$. Now consider the 2-edge-coloured graph $H=G[I_{h(v_1)},\ldots{},I_{h(v_n)}]$ that we obtain by replacing each vertex $v_i$ by an independent set of size $h(v_i)$. Then $H$ is an extension of $G$ and as $G$ is an extension of an M-closed 2-edge-coloured graph so is $H$.
    Observe that for each $i\in [k]$ the   closed alternating trail $T_i$ 
 corresponds to an alternating cycle $C_i$ in $H$ and vice versa (just replace each occurrence of a vertex from $I_{h_i(v)}$ in the cycle by the vertex $v$). 
By Proposition \ref{prop:extcc}, $H$ is colour-connected and since $C_1,C_2,\ldots{},C_k$ form an alternating cycle factor of $H$, it follows from Theorem \ref{thm:extMclHC} that $H$  has an alternating hamiltonian cycle $C$. By contracting each of the sets  $I_{h(v_i)}$ into the vertex $v_i$  we convert $C$ into the desired spanning eulerian subgraph of $G$.
\end{proof}

\begin{lemma}
  \label{lem:2trails}
  Let $G$ be an extension of an M-closed 2-edge-coloured graph and let  $G_1,G_2$ be an eulerian factor of $G$. If $G$ is connected, then $G$ is supereulerian unless the following holds for some $i\in [2]$, every spanning closed alternating trail $T_i$ of $G_i$ and every closed alternating trail $T_{3-i}$ of $G_{3-i}$:
  \begin{itemize}
    \item[(1)] Every vertex of $T_i$ is adjacent to every vertex of $T_{3-i}$.
    \item[(2)] the vertices on $T_i$ alternate between having only edges of colour $c$ to $V(T_{3-i})$ (we call them {\bf $c$-vertices}) and having only edges of colour $c'$ to $V(T_{3-i})$ (we call them {\bf $c'$-vertices}). In particular $T_i$ contains no closed subtrail of odd length.
      \item[(3)] There is no edge of colour $c'$ between two $c$-vertices of $T_i$ and no edge of colour $c$ between two $c'$-vertices of $T_i$.
    \end{itemize}
    \noindent{}In particular, if (1)-(3) hold, then  $G$ is not trail-colour-connected.
\end{lemma}

\begin{proof}
  Let $T_i$ be a spanning closed alternating trail in $G_i$, $i\in [2]$. Let $H$ be the extension of $G$ that we obtain by the procedure used in the proof of Lemma \ref{lem:exMclccE} and let $C_1,C_2$ be the alternating cycles in $H$ that correspond to $T_1$ and $T_2$, respectively. If $H$ is colour-connected, then it follows from Theorem  \ref{thm:extMclHC} that $H$  has an alternating hamiltonian cycle $C$, implying, as in the proof above, that $G$ is supereulerian. Thus we may assume that $H$ and hence also $G$  is not colour-connected.
  Now it follows from the discussion of the proof of Theorem \ref{thm:extMclHC}
  that we have $C_i{\stackrel{c}{\rightarrow}}C_{3-i}$, where $c$ is one of the two colours used to colour $G$. 
  This implies  that back in $G$  (1), (2) and (3) hold for $T_i$ and $T_{3-i}$. The last claim follows from Proposition \ref{prop:Tccext} and the fact that $H$ is not trail-colour-connected (see the remark just after Lemma \ref{lem:dom}).
\end{proof}

Now we are ready to  give a full characterization of those extensions of M-closed 2-edge-coloured graphs which are supereulerian. Recall that trail-colour-connected and colour-connected is not the same thing for M-closed 2-edge-coloured graphs. 

\begin{theorem}
  \label{thm:exMclEchar}
  Let $G$ be an extension of an M-closed  2-edge-coloured graph. Then $G$ is supereulerian if and only if it is trail-colour-connected and has an eulerian factor.
\end{theorem}

\begin{proof}
  If $G$ is supereulerian, then it is trail-connected \AY{and contains an eulerian factor}, so it suffices to consider the other direction. Assume that $G$ is trail-colour-connected with an eulerian factor and let ${\cal G}=G_1,G_2,\ldots{},G_k$, $k\geq 1$ be an eulerian factor of $G$ which is chosen such that $k$ is minimum. If $k=1$ we are done so assume that $k\geq 2$. Let $T_1,\ldots{},T_r$ be arbitrary spanning closed trails in $G_1,G_2,\ldots{},G_k$, respectively and fix a starting vertex $v_{i,1}$ for each trail $T_i$, $i\in [k]$. We will use the notation $T_i{\stackrel{c}{\rightarrow}}T_j$ to denote that (1),(2) and (3) in Lemma \ref{lem:2trails} hold for distinct $i,j\in [k]$, where $T_i$ plays the role of $T_1$ in the Lemma and all edges between $v_{i,1}$ and $V(T_j)$ have colour $c$. We also write $T_i\rightarrow T_j$ if $T_i{\stackrel{a}{\rightarrow}}T_j$ for some $a\in \{c,c'\}$, where the edges of $G$ are coloured by colours $c,c'$.\\
By the minimality of $k$ and Lemma \ref{lem:2trails}, if there is an edge between $G_i$ and $G_j$, then we have $T_i\dom T_j$ or $T_j\dom T_i$. In particular 
there are no two similar vertices $u,v$ which belong to different $G_i$'s (since this would imply that that $u$ and $v$  would both have edges of both colours to the other trail). Thus, using that $G$ is connected and an extension of an M-closed 2-edge-coloured graph,   it is easy to see that there is an edge between $V(T_i)$ and $V(T_j)$ for every choice of $1\leq i<j\leq k$. Hence, by the remark above,  we have $T_i\dom T_j$ or $T_j\dom T_i$ for every choice of $1\leq i<j\leq k$. Let $W$ be the tournament with vertex set $w_1,w_2,\ldots{},w_k$ such that $w_iw_j$ is an arc of $W$ if $T_i\dom T_j$. Suppose first that $W$ contains a cycle. Then it follows from the well known result by Moon \cite{moonCMB9} that a strong tournament is vertex pancyclic, that $W$ has a 3-cycle $w_a\dom w_b\dom w_c\dom w_a$ and hence we have $T_a\dom{}T_b\dom T_c\dom T_a$. In this case we can replace $T_a,T_b,T_c$ by one closed trail as indicated in Figure \ref{fig:3cycle}.

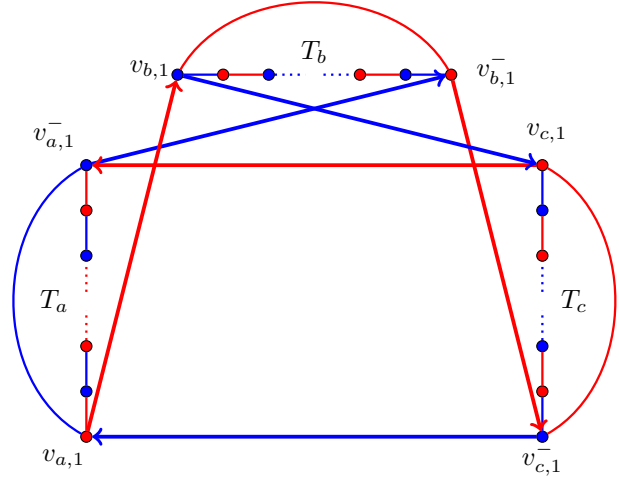
\begin{figure}[H]
  \begin{center}
    \begin{tikzpicture}[scale=0.6]
      \tikzstyle{vertexR}=[circle,draw, top color=red!100, bottom color=red!100, minimum size=7pt, scale=0.6, inner sep=0.5pt]
\tikzstyle{vertexBl}=[circle,draw, top color=blue!100, bottom color=blue!100, minimum size=7pt, scale=0.6, inner sep=0.5pt]
      \tikzstyle{vertexB}=[circle,draw, minimum size=15pt, scale=0.6, inner sep=0.5pt]
      \node(a1) at (0,0) [vertexR] {};
      \node(a2) at (0,1) [vertexBl] {};
      \node(a3) at (0,2) [vertexR] {};
      \node(a4) at (0,4) [vertexBl] {};
      \node(a5) at (0,5) [vertexR] {};
      \node(a6) at (0,6) [vertexBl] {};
      \LineRED {(a1) to (a2);}
      \LineBLUE{(a2) to (a3);}
      \draw [line width=0.03cm, dotted, color=red] (a3) to (0,2.8);
      \draw [line width=0.03cm, dotted, color=red] (0,3.2) to (a4);
      \LineBLUE{(a1) to [out=150,in=210] (a6);}
      \LineRED {(a5) to (a6);}
      \LineBLUE{(a4) to (a5);}

      \node(c1) at (10,0) [vertexBl] {};
      \node(c2) at (10,1) [vertexR] {};
      \node(c3) at (10,2) [vertexBl] {};
      \node(c4) at (10,4) [vertexR] {};
      \node(c5) at (10,5) [vertexBl] {};
      \node(c6) at (10,6) [vertexR] {};
      \LineBLUE{(c1) to (c2);}
      \LineRED{(c2) to (c3);}
      \draw [line width=0.03cm, dotted, color=blue] (c3) to (10,2.8);
      \draw [line width=0.03cm, dotted, color=blue] (10,3.2) to (c4);
      \LineRED{(c4) to (c5);}
      \LineBLUE{(c5) to (c6);}
      \LineRED{(c1) to [out=30,in=-30] (c6);}
      \node (b1) at (2,8) [vertexBl]{};
      \node (b2) at (3,8) [vertexR]{};
      \node (b3) at (4,8) [vertexBl]{};
      \node (b4) at (6,8) [vertexR]{};
      \node (b5) at (7,8) [vertexBl]{};
      \node (b6) at (8,8) [vertexR]{};
      \LineBLUE{(b1) to (b2);}
      \LineRED{(b2) to (b3);}
      \LineBLUE{(b5) to (b6);}
      \LineRED{(b4) to (b5);}
      \draw[line width=0.03cm, dotted, color=blue] (b3) to (4.8,8);
      \draw [line width=0.03cm, dotted, color=blue] (5.2,8) to (b4);
      \LineRED{(b1) to [out=60,in=120] (b6);}
      \draw[->,line width=0.05cm, color=blue] (c1) to (a1);
      \draw[->,line width=0.05cm, color=red] (b6) to (c1);
      \draw[->,line width=0.05cm, color=blue] (a6) to (b6);
      \draw[->,line width=0.05cm, color=red] (c6) to (a6);
      \draw[->,line width=0.05cm, color=blue] (b1) to (c6);
      \draw[->,line width=0.05cm, color=red] (a1) to (b1);
      \node () at (10.7,3) {$T_c$};
      \node () at (-0.7,3){$T_a$};
      \node () at (5,8.5) {$T_b$};
      \node () at (1.4,8.1) {$v_{b,1}$};
      \node () at (9,8.1){$v^-_{b,1}$};
      \node () at (-0.5,-0.5) {$v_{a,1}$};
      \node () at (-0.7,6.7) {$v^-_{a,1}$};
      \node () at (10.1,6.7) {$v_{c,1}$};
      \node () at (10,-0.5) {$v^-_{c,1}$};

   \end{tikzpicture}
  \end{center}
  \caption{An illustration of the case when we have $T_a\dom T_b\dom T_c\dom T_a$. The colours of the vertices denote the colour of all edges from that vertex to the vertices of the trail which it is monochromatic to  (E.g. every edge between a blue vertex of $T_a$ and $V(T_b)$ is blue). The oriented fat edges between the trails indicate a 6-cycle that can be used to merge the three closed trails into one. We obtain the desired trail by starting at $v_{a,1}$, traversing $T_a$, then going from $v_{a,1}$ to $v_{b,1}$, traversing $T_b$, then going from $v_{b,1}$ to $v_{c,1}$, traversing $T_c$ and finally using the arcs $v_{c,1}v^-_{a,1},v^-_{a,1}v^-_{b,1},v^-_{b,1}v^-_{c,1},v^-_{c,1}v_{a,1}$.}\label{fig:3cycle}
  \end{figure}

  Hence we may assume that $W$ is an acyclic (transitive) tournament and that the ordering of ${\cal G}$ is such that $T_i\dom T_j$ whenever $1\leq i<j\leq k$. As $G$ is trail-connected, we must have $k>2$ by Lemma \ref{lem:2trails}.  Let $v$ be a $c$-vertex of $T_1$ with respect to $T_2$. Recall that this means that all edges between $v$ and $V(T_2)$ have colour $c$. If $v$ is a $c$-vertex with respect to $V(T_i)$ for every $1<i\leq k$, then $G$ is not trail-colour-connected as the vertex $v$ has no trail starting with colour $c'$ to any vertex out-side $V(T_1)$. Hence we may assume w.l.o.g. that $v$ is a $c'$-vertex with respect to $V(T_3)$. Now we can merge $T_1,T_2,T_3$ into one closed alternating trail as indicated in Figure \ref{fig:TT3}. This contradicts the minimality of $k$ and the proof is complete.
  \end{proof}

  \begin{figure}[H]
    \begin{center}
      \begin{tikzpicture}[scale=0.6]
        \tikzstyle{vertexB}=[circle,draw, minimum size=7pt, scale=0.6, inner sep=0.5pt]
      \tikzstyle{vertexR}=[circle,draw, top color=red!100, bottom color=red!100, minimum size=7pt, scale=0.6, inner sep=0.5pt]
      \tikzstyle{vertexBl}=[circle,draw, top color=blue!100, bottom color=blue!100, minimum size=7pt, scale=0.6, inner sep=0.5pt]

      \node(a1) at (0,0) [vertexB] {};
      \node(a2) at (0,1) [vertexB] {};
      \node(a3) at (0,2) [vertexB] {};
      \node(a4) at (0,4) [vertexB] {};
      \node(a5) at (0,5) [vertexB] {};
      \node(a6) at (0,6) [vertexB] {};
      \LineRED {(a1) to (a2);}
      \LineBLUE{(a2) to (a3);}
      \draw [line width=0.03cm, dotted, color=red] (a3) to (0,2.8);
      \draw [line width=0.03cm, dotted, color=red] (0,3.2) to (a4);
      \LineBLUEdotted{(a1) to [out=150,in=210] (a6);}    
      \LineRED {(a5) to (a6);}
      \LineBLUE{(a4) to (a5);}

      \node(c1) at (10,0) [vertexB] {};
      \node(c2) at (10,1) [vertexB] {};
      \node(c3) at (10,2) [vertexB] {};
      \node(c4) at (10,4) [vertexB] {};
      \node(c5) at (10,5) [vertexB] {};
      \node(c6) at (10,6) [vertexB] {};
      \LineBLUE{(c1) to (c2);}
      \LineRED{(c2) to (c3);}
      \draw [line width=0.03cm, dotted, color=blue] (c3) to (10,2.8);
      \draw [line width=0.03cm, dotted, color=blue] (10,3.2) to (c4);
      \LineRED{(c4) to (c5);}
      \LineBLUE{(c5) to (c6);}
      \LineREDdotted{(c1) to [out=30,in=-30] (c6);}   
      \node (b1) at (2,8) [vertexB]{};
      \node (b2) at (3,8) [vertexB]{};
      \node (b3) at (4,8) [vertexB]{};
      \node (b4) at (6,8) [vertexB]{};
      \node (b5) at (7,8) [vertexB]{};
      \node (b6) at (8,8) [vertexB]{};
      \LineBLUE{(b1) to (b2);}
      \LineRED{(b2) to (b3);}
      \LineBLUE{(b5) to (b6);}
      \LineRED{(b4) to (b5);}
      \draw[line width=0.03cm, dotted, color=blue] (b3) to (4.8,8);
      \draw [line width=0.03cm, dotted, color=blue] (5.2,8) to (b4);
      \LineRED{(b1) to [out=60,in=120] (b6);}
      \draw[line width=0.06cm, color=blue] (b1) to (a1);
      \draw[line width=0.06cm, color=blue] (b1) to (a6);
      \draw[line width=0.06cm, color=red] (b1) to (c1);
      \draw[line width=0.06cm, color=red] (b1) to (c6);
      \node () at (1.3,8.1) {$v_{1,1}$};

      \node () at (10.7,3) {$T_2$};
      \node () at (-0.7,3){$T_3$};
      \node () at (5,8.5) {$T_1$};

  \end{tikzpicture}
  \end{center}
  \caption{Showing how to merge the trails $T_1,T_2,T_3$ into one trail. Starting from $v_{1,1}$ first pick up all vertices of $V(T_2)$ and return to $v_{1,1}$; then pick up all vertices of $V(T_3)$ and return to $v_{1,1}$; finally traverse $T_1$ starting from $v_{1,1}$.}\label{fig:TT3}
\end{figure}
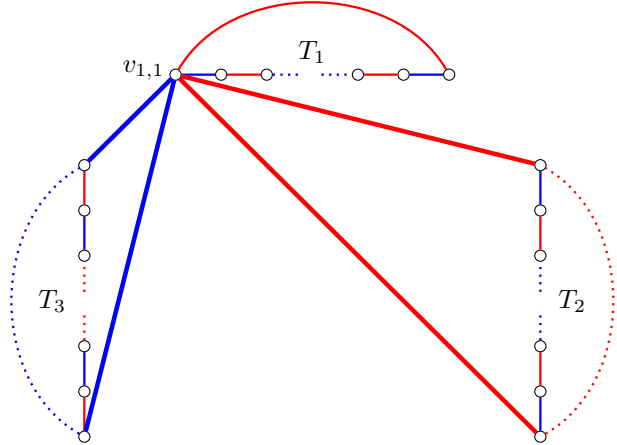

All arguments in the proofs above are algorithmic so combining this with Lemma \ref{lem:findEF} we get the following.

\begin{corollary}
  There exists a polynomial algorithm ${\cal B}$ which given  a graph $G$ which is an extension of an M-closed 2-edge-coloured multigraph, either returns a spanning closed alternating trail of $G$ or provides a certificate that $G$ has no such trail, because it either has no  eulerian factor or is not  trail-colour-connected.
\end{corollary}

By Lemma \ref{lem:CMGtcciffcc}, trail-colour-connectivity coincides with colour-connectivity for extended 2-edge-coloured complete graphs. Hence we get the following characterization of supereulerian extensions of 2-edge-coloured complete graphs

\begin{corollary}\label{mainthm}
An  extended 2-edge-coloured complete graph is supereulerian if and only if it has an eulerian factor and is colour-connected.
\end{corollary}

\section{Complexity for general 2-edge-coloured graphs}

\begin{theorem}\label{thmsesnpc}
 It is NP-complete to decide if a 2-edge-coloured graph is supereulerian.
\end{theorem}

\begin{proof}
 We show how to reduce  the problem of deciding if a 2-edge-coloured graph has an alternating hamiltonian cycle to the problem of deciding if a 2-edge-coloured graph is supereulerian. Let $G$ be a 2-edge-coloured graph. We create the graph $G'$ as follows:
 \begin{itemize}
  \item for every vertex $v\in V(G)$, add vertices $v_{r}$ and $v_{b}$.
  \item replace every red edge $uv$ by $u_{r}v_{r}$ and every blue edge $uv$ by $u_{b}v_{b}$.
  \item for every vertex $v\in V(G)$, add a blue edge $v_{r}v$ and a red edge $vv_{b}$.
 \end{itemize}
The construction of $G'$ is illustrated in Figure \ref{figsesnpc}. Since the vertices $v_r$ only have one incident blue edge, they can only be used once in an eulerian subgraph. Hence, when a spanning closed alternating trail reaches a vertex $v_r$, it has to go to $v$ and then $v_b$ and cannot go back to any of this vertices again. This implies that if a spanning eulerian subgraph of $G'$ exists, it immediately provides a hamiltonian cycle in $G$.

Conversely, if a hamiltonian cycle exists in $G$, we replace every vertex $v$ by $v_rvv_b$ if $v$ is reached with a red edge and by $v_bvv_r$ otherwise and we obtain a spanning eulerian subgraph of $G'$.\qedhere

 \begin{figure}[!h]
 \begin{subfigure}{0.48\linewidth}

\centering{\includegraphics{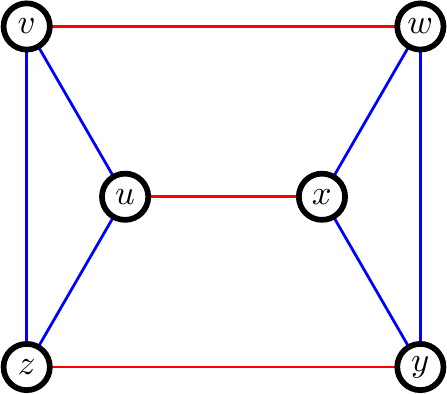}}
\caption{A non-hamiltonian graph $G$.}
\end{subfigure}
 \begin{subfigure}{0.48\linewidth}
 
\centering{\includegraphics{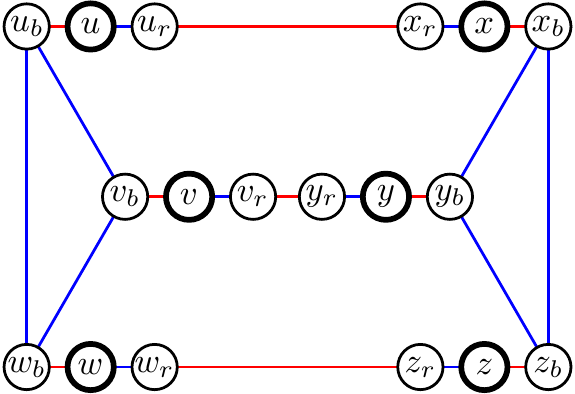}}
\caption{The associated graph $G'$ has no spanning eulerian subgraph.}
\end{subfigure}
\caption{}\label{fig:NPCfig}
\label{figsesnpc}
\end{figure}

 \end{proof}

 The construction that we  used above  may not lead to a 2-edge-coloured graph with an  eulerian factor, but we can modify it by replacing every vertex $v$ by the gadget $g_v$ depicted in Figure \ref{figgadget}. We still replace every red edge $uv$ by $u_rv_r$ and blue edge $uv$ by $u_bv_b$. As previously, the vertices $v_r$ and $v_b$ can only be used once in a spanning eulerian subgraph because they have only one incident blue and red edge respectively. We find that the spanning eulerian subgraphs of $G'$ are exactly the hamiltonian cycles of $G$ where we replace vertices $v$ by $v_bv_1v_2v_3v_4v_1v_r$ or $v_rv_1v_4v_3v_2v_1v_b$. Note that the union for $v\in V(G)$ of the $v_bv_1v_rv_4v_3v_2v_b$ provide an eulerian factor in $G'$.\qedhere
 
  \begin{figure}[!h]
\centering{\includegraphics{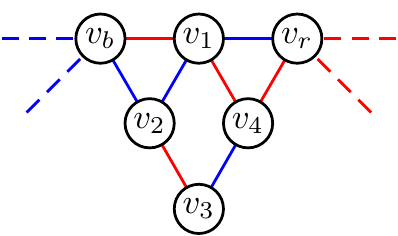}}
\caption{The gadget $g_v$.}
\label{figgadget}
\end{figure}

\section{Further  generalizations of 2-edge-coloured complete graphs}

Complete bipartite graphs are a subclass of the class of complete multipartite graphs. As we saw in the end of Section \ref{sec:supereuler}, for 2-edge-coloured complete bipartite graphs we can rely on results on supereulerian  bipartite tournaments to classify supereulerian 2-edge-coloured complete bipartite graphs. For general 2-edge-coloured complete multipartite graphs, we have no correspondence similar to the BB-correspondence.

\begin{problem}
  What is the completely of deciding whether a 2-edge-coloured complete multipartite graph has an alternating hamiltonian cycle? Is there a good characterization?
\end{problem}

The following is an easy consequence of Lemma \ref{lem:CMGtcciffcc}.

\begin{prop}\label{prop:CMGEulercc}
If a 2-edge-coloured complete multipartite graph $G$ has a spanning closed alternating trail, then $G$ is colour-connected.
\end{prop}

\begin{prop}
\label{prop:CMGccnoEuler}
There exists infinitely many non-supereulerian 2-edge-coloured complete multipartite graphs which are colour-connected and have an alternating cycle factor.
\end{prop}

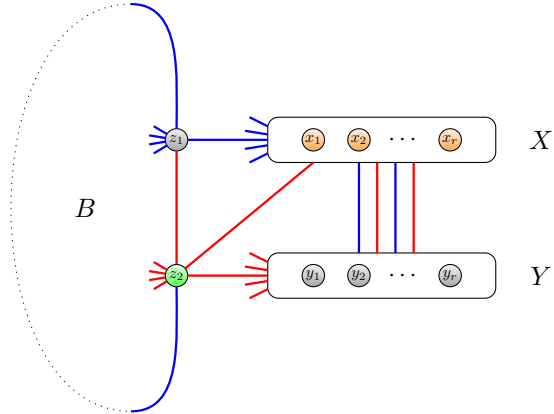
\begin{figure}[H]
\begin{center}
\begin{tikzpicture}[scale=0.6]

\node (z1) at (3,7) [vertexX] {$z_1$};  
\node (z2) at (3,4) [vertexY] {$z_2$};  

\LineBLUE{(z1) -- (2.5,7.3);}
\LineBLUE{(z1) -- (2.4,7.1);}
\LineBLUE{(z1) -- (2.4,6.9);}
\LineBLUE{(z1) -- (2.5,6.7);}

\LineRED{(z2) -- (2.5,4.3);}
\LineRED{(z2) -- (2.4,4.1);}
\LineRED{(z2) -- (2.4,3.9);}
\LineRED{(z2) -- (2.5,3.7);}

\LineBLUE{(z1) to [out=90, in=0] (2,10);}
\LineRED{(z1) -- (z2);}
\LineBLUE{(z2) to [out=270, in=0] (2,1);}
\draw[dotted] (2,10) to [out=180, in=180] (2,1); 

\draw [rounded corners] (5,6.5) rectangle (10,7.5); 
\draw [rounded corners] (5,3.5) rectangle (10,4.5);  

\LineRED{(5,4.3) -- (4.6,4.5);}
\LineRED{(5,4.12) -- (4.5,4.2);}
\LineRED{(5,3.88) -- (4.5,3.8);}
\LineRED{(5,3.7) -- (4.6,3.5);}

\LineBLUE{(5,7.3) -- (4.6,7.5);}
\LineBLUE{(5,7.12) -- (4.5,7.2);}
\LineBLUE{(5,6.88) -- (4.5,6.8);}
\LineBLUE{(5,6.7) -- (4.6,6.5);}

\node () at (1,5.5) {$B$};
\node () at (11,7) {$X$};
\node () at (11,4) {$Y$};
\node (x1) at (6,7) [vertexZ] {$x_1$};
\node (x2) at (7,7) [vertexZ] {$x_2$};
\node (xr) at (9,7) [vertexZ] {$x_r$};
\node [scale=0.9] at (8,7) {$\cdots$}; 

\node (y1) at (6,4) [vertexX] {$y_1$};
\node (y2) at (7,4) [vertexX] {$y_2$};
\node (yr) at (9,4) [vertexX] {$y_r$};
\node [scale=0.9] at (8,4) {$\cdots$};

\LineRED{(z2) -- (5,4);}
\LineRED{(z2) -- (6,6.5);}
\LineBLUE{(z1) -- (5,7);}

\LineBLUE{(7,4.5) -- (7,6.5);}
\LineRED{(7.4,4.5) -- (7.4,6.5);}
\LineBLUE{(7.8,4.5) -- (7.8,6.5);}
\LineRED{(8.2,4.5) -- (8.2,6.5);}

\end{tikzpicture} 
\end{center}
\caption{An infinite family ${\cal G}$ of  2-edge-coloured complete 3-partite graphs which are colour-connected and have an alternating cycle factor but are not supereulerian. The 3 partite sets are indicated in colours grey, orange and green. The left part is a 2-edge-coloured complete bipartite graph with an alternating  hamiltonian cycle indicated. The grey vertices belong to the same colour class as the vertices in $Y$. Every vertex in $X$ is connected by blue edges to all vertices of $B$ and every vertex in $Y$ is connected by red edges to the green vertices in $B$. The two special vertices $z_1,z_2$ are joined by a blue edge and all other edges incident to $z_1$ ($z_2$) in $B$ are blue (red). The complete bipartite subgraph induced by 
$X\cup Y$ has an alternating hamiltonian cycle $x_1y_1x_2y_2\ldots{}x_ry_rx_1$ and all other edges can be coloured arbitrarily red or blue. Se Figure \ref{fig:CMGex} below for a specific example of a graph in ${\cal G}$. }\label{fig:CMGfig1}
\end{figure}

\begin{figure}[H]
\begin{center}
\begin{tikzpicture}[scale=0.5]

\node (z1) at (3,10) [vertexX] {$z_1$};  
\node (z2) at (3,7) [vertexY] {$z_2$};  
\node (z3) at (3,4) [vertexX] {$z_3$};  
\node (z4) at (3,1) [vertexY] {$z_4$};  

\LineBLUE{(z1) to [out=190, in=170] (z4);}
\LineRED{(z1) -- (z2);}
\LineBLUE{(z2) -- (z3);}
\LineRED{(z3) -- (z4);}

\node (x1) at (10,10) [vertexZ] {$x_1$};
\node (y1) at (10,7) [vertexX] {$y_1$};
\node (x2) at (10,4) [vertexZ] {$x_2$};
\node (y2) at (10,1) [vertexX] {$y_2$};

\LineBLUE{(x1) -- (z1);}
\LineBLUE{(x1) -- (z3);}
\LineBLUE{(x1) -- (z4);}
\LineBLUE{(x2) -- (z1);}
\LineBLUE{(x2) -- (z3);}
\LineBLUE{(x2) -- (z4);}

\LineRED{(x1) -- (z2);}
\LineRED{(x2) -- (z2);}

\LineRED{(y1) -- (z2);}
\LineRED{(y2) -- (z2);}
\LineRED{(y1) -- (z4);}
\LineRED{(y2) -- (z4);}

\LineBLUE{(x1) to [out=350, in=10] (y2);}
\LineRED{(x1) -- (y1);}
\LineBLUE{(y1) -- (x2);}
\LineRED{(x2) -- (y2);}

\end{tikzpicture}
\end{center}
\caption{}\label{fig:CMGex}
\end{figure}
\begin{proof}
Let $G$ be a 2-edge-coloured complete 3-partite graph from the infinite family described in Figure \ref{fig:CMGfig1}. 
Every  alternating spanning eulerian subgraph, $H$, must use the edge $z_1z_2$ as it is the only red edge incident with $z_1$.
As there is only one blue edge incident with $z_2$, the edge $z_1z_2$ is the only red incident with $z_2$ in $H$.
Let $uv$ be any edge in $H$ from $ u \in V(G)\setminus (X \cup Y)$ to $ v\in X \cup Y$. 
Either $uv$ is blue and $v \in X$ or $uv$ is red and $v \in Y$ (by the above).
Without loss of generality assume that $uv$ is blue and $v \in X$.
It is not difficult to see that the successor of $v$ in $H$ lies in $Y$ and the successor of this vertex is back in $X$, etc.
As we cannot return to $V(G)\setminus (X \cup Y)$, we obtain a contradiction.

\end{proof}

Despite the existence of the class ${\cal G}$ we still believe that one can recognize supereulerian 2-edge-coloured complete multipartite graphs in polynomial time.
\begin{conjecture}
  There exists a polynomial algorithm for deciding whether a 2-edge-coloured complete multipartite graph is supereulerian.
  \end{conjecture}


\end{document}